\documentclass [twoside,reqno, 12pt] {amsart}

\usepackage{hyperref}

\usepackage{amsfonts}
\usepackage{amssymb}
\usepackage{a4}
\usepackage{color}
\usepackage{amscd,amssymb,amsfonts}

\usepackage{picture,xcolor}

\usepackage{etoolbox}

\usepackage{tikz}

\newtheorem{thm}{Theorem}[section]

\newtheorem{prop}[thm]{Proposition}

\newtheorem{rem}[thm]{Remark}

\theoremstyle{definition}

\numberwithin{equation}{section}

\newcommand{\C}{\mathbb{C}}

\newcommand{\R}{\mathbb{R}}

\newcommand{\supp}{\operatorname{supp}}

\def \bfo {\begin {eqnarray*} }
\def \efo {\end {eqnarray*} }
\def \ba {\begin {eqnarray*} }
\def \ea {\end {eqnarray*} }
\def \beq {\begin {eqnarray}}
\def \eeq {\end {eqnarray}}
\def \supp {\hbox{supp }}

\def \p {\partial}

\def \bfo {\begin {eqnarray*} }
\def \efo {\end {eqnarray*} }
\def \ba {\begin {eqnarray*} }
\def \ea {\end {eqnarray*} }
\def \beq {\begin {eqnarray}}
\def \eeq {\end {eqnarray}}
\def \supp {\hbox{supp }}

\def \p {\partial}

\UseRawInputEncoding

\begin{document}

\title[Nonlinear magnetic Schr\"odinger equations on complex manifolds]{A remark on inverse problems for nonlinear magnetic Schr\"odinger equations on complex manifolds}

\author[Krupchyk]{Katya Krupchyk}
\address
        {K. Krupchyk, Department of Mathematics\\
University of California, Irvine\\
CA 92697-3875, USA }

\email{katya.krupchyk@uci.edu}

\author[Uhlmann]{Gunther Uhlmann}

\address
       {G. Uhlmann, Department of Mathematics\\
       University of Washington\\
       Seattle, WA  98195-4350\\
       USA\\
        and Institute for Advanced Study of the Hong Kong University of Science and Technology}
\email{gunther@math.washington.edu}

\author[Yan]{Lili Yan}

\address
        {L. Yan, Department of Mathematics\\
         University of California, Irvine\\ 
         CA 92697-3875, USA }

\email{liliy6@uci.edu}

\maketitle
\begin{abstract}
We show that the knowledge of the Dirichlet--to--Neumann map for a nonlinear magnetic Schr\"odinger operator  on the boundary of a compact complex manifold, equipped with a  K\"ahler metric and admitting sufficiently many global holomorphic functions, determines the nonlinear magnetic and electric potentials uniquely.

\end{abstract}

\section{Introduction}

Let $M$ be an $n$--dimensional compact complex manifold with $C^\infty$ boundary, equipped with a K\"ahler metric $g$. Consider the nonlinear magnetic Schr\"odinger operator 
\[
L_{A,V}u=d^*_{\overline{A(\cdot, u)}}d_{A(\cdot, u)}u +V(\cdot, u),
\]
acting on $u\in C^\infty(M)$. Here the nonlinear magnetic $A: M\times \C\to T^*M\otimes\C$  and electric $V:M\times \C\to \C$ potentials are assumed to satisfy the following conditions:  
\begin{itemize}
\item[(i)] the map $\C\ni w\mapsto A(\cdot, w)$ is holomorphic with values in $C^{\infty}(M,T^*M\otimes\C)$, 
\item[(ii)] $A(z,0)=0$ for all $z\in M$, 
\item[(iii)] the map $\C\ni w\mapsto  V(\cdot, w)$ is holomorphic with values in $C^\infty(M)$, 
\item[(iv)] $V(z,0)=\p_w V(z,0)=0$ for all $z\in M$. 
 \end{itemize}
Thus, $A$ and $V$ can be expanded into power series
\begin{equation}
\label{eq_int_1_1}
A(z,w)=\sum_{k=1}^\infty A_k(z) \frac{w^k}{k!}, \quad V(z,w)=\sum_{k=2}^\infty V_k(z)\frac{w^k}{k!},
\end{equation}
converging in $C^{\infty}(M,T^*M\otimes\C)$ and $C^{\infty}(M)$ topologies, respectively. Here 
\[
A_k(z):=\p_w^k A(z,0)\in C^{\infty}(M,T^*M\otimes\C),\quad V_k(z):=\p_w^kV(z,0)\in C^\infty(M). 
\]
We write $T^*M\otimes \C$ for the complexified cotangent bundle of $M$, 
\begin{equation}
\label{eq_int_1_1_2}
d_{A(\cdot,w)}=d+iA(\cdot,w): C^\infty(M)\to C^\infty(M, T^*M\otimes \C),\quad w\in \C,
\end{equation}
where $d: C^\infty(M)\to C^\infty(M, T^*M\otimes \C)$ is the de Rham differential, and $d_{A(\cdot,w)}^*: C^\infty(M, T^*M\otimes \C)\to C^\infty(M)$  is the formal $L^2$--adjoint of $d_{A(\cdot,w)}$ taken with respect to the K\"ahler metric $g$. 

It is established in \cite[Appendix B]{KU20} that under the assumptions (i)-(iv), there exist $\delta>0$ and $C>0$ such that for any $f\in B_\delta(\p M):=\{f\in C^{2,\alpha}(\p M): \|f\|_{C^{2,\alpha}(\p M)}<\delta\}$, $0<\alpha<1$, the Dirichlet problem for the nonlinear magnetic Schr\"odinger operator
\begin{equation}
\label{eq_int_1_2}
\begin{cases}
L_{A,V}u=0 & \text{in}\quad M^{\text{int}},\\
u|_{\p M}=f,
\end{cases}
\end{equation}
has a unique solution $u=u_f\in C^{2,\alpha}(M)$ satisfying $\|u\|_{C^{2,\alpha}(M)}<C\delta$. Here $C^{2,\alpha}(M)$ and $C^{2,\alpha}(\p M)$ stand for the standard H\"older spaces of functions on $M$ and $\p M$, respectively, and $M^{\text{int}}=M\setminus \p M$ stands for  the interior of $M$.  Associated to \eqref{eq_int_1_2}, we introduce the Dirichlet--to--Neumann map 
\begin{equation}
\label{eq_int_1_3}
\Lambda_{A,V}f=\p_\nu u_f|_{\p M}, \quad f\in B_\delta(\p M),
\end{equation}
where $\nu$ is the unit outer normal to the boundary of $M$. 

The inverse boundary problem for  the nonlinear magnetic Schr\"odinger operator  that we are interested in asks whether the knowledge of the Dirichlet--to--Neumann map $\Lambda_{A,V}$ determines the nonlinear magnetic $A$ and electric $V$ potentials in $M$. Such inverse problems have been recently studied in \cite{KU20} in the case of conformally transversally anisotropic manifolds and in \cite{Lai_Ting} and \cite{Ma} in the case of partial data in the Euclidean space and on Riemann surfaces, respectively.  

To state our result, following \cite{Guillarmou_Salo_Tzou_2019}, we assume that the manifold $M$ satisfies the following additional assumptions: 
\begin{itemize}
\item[(a)]  $M$ is holomorphically separable in the sense that if $x,y\in M$ with $x\ne y$, there is some $f\in \mathcal{O}(M):=\{f\in C^\infty(M): f \text{ is holomorphic in } M^{\text{int}}\}$ such that $f(x)\ne f(y)$, 
\item[(b)] $M$ has local charts given by global holomorphic functions in the sense that for every $p\in M$ there exist $f_1,\dots, f_n\in \mathcal{O}(M)$ which form a complex coordinate system near $p$. 
\end{itemize}
As explained in \cite{Guillarmou_Salo_Tzou_2019}, examples of complex manifolds satisfying all of the assumptions above including  (a) and (b) are as follows:
\begin{itemize}
\item any compact $C^\infty$ subdomain of a Stein manifold, equipped  with a K\"ahler metric,
\item any compact $C^\infty$ subdomain of a complex submanifold of $\C^N$, equipped with a K\"ahler metric, 
\item any compact $C^\infty$ subdomain of a complex coordinate neighborhood on a K\"ahler manifold. 
\end{itemize}

The main result of this note is as follows. 
\begin{thm}
\label{thm_main}
Let $M$ be an $n$--dimensional compact complex manifold with $C^\infty$ boundary, equipped with a K\"ahler metric $g$, satisfying assumptions (a) and (b). Let $A^{(1)}, A^{(2)}: M\times \C\to T^*M\otimes\C$  and $V^{(1)}, V^{(2)}:M\times \C\to \C$ be such that the assumptions \emph{(i)}--\emph{(iv)} hold. If $\Lambda_{A^{(1)}, V^{(1)}}=\Lambda_{A^{(2)}, V^{(2)}}$  then $A^{(1)}=A^{(2)}$ and 
$V^{(1)}=V^{(2)}$ in $M\times\C$. 
\end{thm}

\begin{rem}
Theorem \ref{thm_main} in the case of a semilinear Schr\"odinger operator, i.e. when $A=0$, was obtained in  \cite{Ma_Tzou_2021}.
\end{rem}

\begin{rem}
The corresponding inverse problems for the linear Schr\"odinger operator $-\Delta_g+V_0$, $V_0\in C^\infty(M)$, as well as for the linear magnetic Schr\"odinger operator $d^*_{\overline{A_0}}d_{A_0}+V_0$, $A_0\in C^\infty(M,T^*M\otimes\C)$, in the geometric setting of Theorem \ref{thm_main}  are  open.  Theorem \ref{thm_main} can be viewed as a manifestation of the phenomenon, discovered in \cite{Kurylev_Lassas_Uhlmann_2018}, that the presence of nonlinearity may help to solve inverse problems. We refer to \cite{Guillarmou_Salo_Tzou_2019} for the solution of the linearized inverse problem for the linear Schr\"odinger operator in this geometric setting, and would like to emphasize that our proof of  Theorem \ref{thm_main} is based crucially on this result.  We also refer to \cite{Guillarmou_Tzou_2011_proc},  \cite{Guillarmou_Tzou_2011_partial}, \cite{Guillarmou_Tzou_2011} for solutions to inverse boundary problems for the linear Schr\"odinger and magnetic Schr\"odinger operators on Riemann surfaces.  
\end{rem}

\begin{rem}  
The known results for the inverse boundary problem for the linear Schr\"odinger and magnetic  Schr\"odinger operators on Riemannian manifolds of dimension $\ge 3$ with boundary beyond the Euclidean ones, see \cite{Sylvester_Uhlmann_1987}, \cite{Nakamura_Sun_Uhlmann}, \cite{Krup_Uhlmann_2014},  and real analytic ones, see \cite{Lassas_Uhlmann}, \cite{Lassas_Taylor_Uhlmann}, \cite{Lee_Uhlmann}, all require a certain conformal symmetry of the manifold as well as some additional assumptions about the injectivity of geodesic ray transforms, see \cite{DKSaloU_2009},  \cite{DKurylevLS_2016}, \cite{Cekic}, \cite{Krup_Uhlmann_magn_2018}. The known results for inverse problems for the nonlinear Schr\"odinger operators $L_{0, V}$  \cite{Feizmohammadi_Oksanen}, \cite{LLLS}, and nonlinear magnetic Schr\"odinger operators $L_{A,V}$  \cite{KU20}  still require the same conformal symmetry of the manifold, while the injectivity of the geodesic transform is no longer needed.  
\end{rem}

Note that the need to require a certain conformal symmetry of the manifold in all of the known results in dimensions $n\ge 3$ is to due to the existence of limiting Carleman weights on such manifolds, see  \cite{DKSaloU_2009}, which are crucial for the construction of complex geometric optics solutions used for solving inverse problems for elliptic PDE since the fundamental work \cite{Sylvester_Uhlmann_1987}.  However, it is shown in \cite{LS_2012}, \cite{Angulo_2017} that a generic manifold of dimension $n\ge 3$ does not admit limiting Carleman weights. 

\begin{rem}
As in \cite[Theorem 1.1]{Guillarmou_Salo_Tzou_2019}, manifolds considered in Theorem \ref{thm_main} need not admit limiting Carleman weights. For example, it was established in \cite{Angulo_Faraco_Guijarro_Ruiz_2016}  that  ${\C}P^2$ with the Fubini-Study metric $g$ does not admit a limiting Carleman weight near any point. However,   $({\C}P^2, g)$ 
 is a K\"ahler manifold, and as explained in \cite{Guillarmou_Salo_Tzou_2019}, compact $C^\infty$ subdomains of it provide examples of manifolds where Theorem \ref{thm_main} applies. 
\end{rem}

\begin{rem} 
In contrast to the inverse boundary problem for the linear magnetic Schr\"odinger equation, where one can determine the magnetic potential up to a gauge transformation only, see for example \cite{Nakamura_Sun_Uhlmann},
\cite{Krup_Uhlmann_2014}, in Theorem \ref{thm_main} the unique determination of both nonlinear magnetic and electric potentials is achieved.  This is due to our assumptions (ii) and (iv) which lead to the first order linearization of the nonlinear magnetic Schr\"odinger equation given by $-\Delta_g u=0$ rather than by the linear magnetic Schr\"odinger equation, see also \cite{KU20} for a similar unique determination in the case of conformally transversally anisotropic manifolds. 
\end{rem}

Let us finally mention that inverse problems for the semilinear Schr\"odinger operators and for nonlinear conductivity equations have been investigated intensively recently, see  for example  \cite{Feizmohammadi_Oksanen},  \cite{LLLS_partial}, \cite{LLLS}, \cite{LLinSTyni},  \cite{KU20_remark}, \cite{KU20_MRL}, and  \cite{CFKKU_2021}, \cite{KKU_conduc_partial}, \cite{CFa20}, \cite{CFb20}, \cite{MU20}, \cite{Sh19}, respectively.

Theorem \ref{thm_main} is a direct consequence of the main result of  \cite{Guillarmou_Salo_Tzou_2019}, combined with some boundary determination results of \cite{Ma_Tzou_2021} and of Appendix \ref{app_boundary_determination}, as well as the higher order linearization procedure introduced in \cite{Kurylev_Lassas_Uhlmann_2018} in the hyperbolic case, and in \cite{Feizmohammadi_Oksanen}, \cite{LLLS} in the elliptic case.  We refer to \cite{Isakov_93} where the method of a first order linearization was pioneered in the study of inverse problems for nonlinear PDE, and to \cite{Assylbekov_Zhou},   \cite{CNV_2019}, \cite{Sun_96}, and \cite{Sun_Uhlm_97} where a second order linearization was  successfully exploited. 
The crucial fact used in  the proof of the main result of  \cite{Guillarmou_Salo_Tzou_2019}, indispensable for our Theorem \ref{thm_main},  is that both holomorphic and antiholomorphic functions are harmonic on K\"ahler manifolds. The assumptions (a) and (b) in Theorem \ref{thm_main} are needed as they are used in  \cite{Guillarmou_Salo_Tzou_2019} to construct suitable  holomorphic and antiholomorphic functions by extending the two dimensional arguments of \cite{Bukhgeim_2008} and \cite{Guillarmou_Tzou_2011_partial} to the case of higher dimensional complex manifolds. 

The plan of the note is as follows.  The proof of Theorem \ref{thm_main} is given in Section \ref{sec_proof}. Appendix \ref{app_boundary_determination} contains the boundary determination result needed in the proof of Theorem \ref{thm_main}.

\section{Proof of Theorem \ref{thm_main}}
\label{sec_proof}

First using that $d_A^*=d^*-i\langle \overline{A}, \cdot\rangle_g$ and \eqref{eq_int_1_1_2}, 
we write the nonlinear magnetic Sch\"odinger operator $L_{A,V}$ as follows,
\begin{align*}
L_{A,V}u&=d^*_{\overline{A(\cdot, u)}}d_{A(\cdot, u)}u+V(\cdot, u)\\
&=-\Delta_gu+d^*(i A(\cdot, u)u)-i\langle A(\cdot, u), du\rangle_g+ \langle A(\cdot, u), A(\cdot, u)\rangle_g u+V(\cdot, u),
\end{align*} 
for $u\in C^\infty(M)$. Here  $\langle \cdot, \cdot\rangle_g$ is the pointwise scalar product in the space of $1$-forms induced by the Riemannian metric $g$, compatible with the K\"ahler structure.

Using the $m$th order linearization of the Dirichlet--to--Neumann map $\Lambda_{A,V}$  and induction on $m=2,3,\dots$, we shall show that the coefficients $A_{m-1}$ and $V_m$ in \eqref{eq_int_1_1} can all be recovered from $\Lambda_{A,V}$.

First, let $m=2$ and let us proceed to carry out a second order linearization of the Dirichlet--to--Neumann map. To that end,  let $f_1, f_2\in C^{\infty}(\p M)$ and let $u_j=u_j(x,\varepsilon)\in C^{2,\alpha}(M)$ be the unique small solution of the following Dirichlet problem,
\begin{equation}
\label{eq_2_1}
\begin{cases}
-\Delta_g u_j+id^*(\sum_{k=1}^\infty A^{(j)}_k(x)\frac{u_j^k}{k!} u_j)-i \langle \sum_{k=1}^\infty A^{(j)}_k(x)\frac{u_j^k}{k!}, du_j\rangle_g\\
 \, \quad   + \langle \sum_{k=1}^\infty A^{(j)}_k(x)\frac{u_j^k}{k!}, \sum_{k=1}^\infty A^{(j)}_k(x)\frac{u_j^k}{k!}\rangle_g u_j 
 +\sum_{k=2}^\infty V^{(j)}_k(x)\frac{u_j^k}{k!}=0 \text{ in }M^{\text{int}},\\
u_j=\varepsilon_1f_1+\varepsilon_2f_2  \text{ on }\p M,
\end{cases}
\end{equation}
for $j=1,2$. It was established in \cite[Appendix B]{KU20} that  for all $|\varepsilon|$ sufficiently small, the solution $u_j(\cdot,\varepsilon)$ depends holomorphically on $\varepsilon=(\varepsilon_1,\varepsilon_2)\in \text{neigh}(0, \C^2)$.  Applying the operator $\p_{\varepsilon_l}|_{\varepsilon=0}$, $l=1,2$, to \eqref{eq_2_1} and using that $u_j(x,0)=0$, we get  
\[
\begin{cases}
-\Delta_g v_j^{(l)}=0 & \text{in } M^{\text{int}},\\
v_j^{(l)}=f_l & \text{on }\p M,
\end{cases}
\]
where $v_j^{(l)}=\p_{\varepsilon_l}u_j|_{\varepsilon=0}$. By the uniqueness and the elliptic regularity, it follows that $v^{(l)}:=v_1^{(l)}=v_2^{(l)}\in C^{\infty} (M)$, $l=1,2$.  Applying $\p_{\varepsilon_1}\p_{\varepsilon_2}|_{\varepsilon=0}$ to \eqref{eq_2_1}, we obtain the second order linearization,
\begin{equation}
\label{eq_2_2}
\begin{cases}
-\Delta_g w_j +2i d^*(A_1^{(j)}v^{(1)}v^{2})- i \langle A_1^{(j)}, d(v^{(1)}v^{(2)})\rangle_g +V^{(j)}_2 v^{(1)}v^{(2)} =0  \text{ in } M^{\text{int}},\\
w_j=0  \text{ on }\p M,
\end{cases}
\end{equation}
where $w_j=\p_{\varepsilon_1}\p_{\varepsilon_2}u_j|_{\varepsilon=0}$, $j=1,2$. Using that 
\begin{equation}
\label{eq_2_2_1}
d^*(Bv)=(d^*B)v -\langle B, dv\rangle_g,
\end{equation}
for any $B\in C^\infty(M, T^*M\otimes \C)$ and $v\in C^\infty(M)$, \eqref{eq_2_2} implies that 
\begin{equation}
\label{eq_2_3}
\begin{cases}
-\Delta_g w_j - 3 i \langle A_1^{(j)}, d(v^{(1)}v^{(2)})\rangle_g +(2id^*(A_1^{(j)})+V^{(j)}_2) v^{(1)}v^{(2)} =0  \text{ in } M^{\text{int}},\\
w_j=0  \text{ on }\p M,
\end{cases}
\end{equation}
$j=1,2$. 
The equality $\Lambda_{A^{(1)}, V^{(1)}}(\varepsilon_1f_1+\varepsilon_2f_2)=\Lambda_{A^{(2)}, V^{(2)}}(\varepsilon_1f_1+\varepsilon_2f_2)$ yields that $\p_\nu u_1|_{\p M}= \p_\nu u_2|_{\p M}$, and  hence, $\p_\nu w_1|_{\p M}= \p_\nu w_2|_{\p M}$.  Multiplying the difference of two equations in \eqref{eq_2_3} by a harmonic function $v^{(3)}\in C^\infty(M)$, integrating over $M$ and using Green's formula, we obtain that  
\begin{equation}
\label{eq_2_4}
\int_M \big(3i \langle A, d(v^{(1)}v^{(2)})\rangle_g v^{(3)}-(2i d^*(A) +V) v^{(1)}v^{(2)}v^{(3)}  \big)dV_g=0,
\end{equation}
valid for all harmonic functions $v^{(l)}\in C^{\infty}(M)$, $l=1,2,3$. Here $A=A_1^{(1)}-A_1^{(2)}$ and $V=V_2^{(1)}-V_2^{(2)}$.  Interchanging $v^{(3)}$ and $v^{(1)}$ in \eqref{eq_2_4}, we also have 
\begin{equation}
\label{eq_2_5}
\int_M \big(3i \langle A, d(v^{(3)}v^{(2)})\rangle_g v^{(1)}-(2i d^*(A) +V) v^{(1)}v^{(2)}v^{(3)}  \big)dV_g=0. 
\end{equation}
Subtracting \eqref{eq_2_5} from \eqref{eq_2_4} and letting $v^{(3)}=1$, we get 
\begin{equation}
\label{eq_2_6}
\int_M  \langle A, dv^{(1)}\rangle_g v^{(2)}dV_g=0, 
\end{equation}
for all harmonic functions $v^{(1)}, v^{(2)}\in C^{\infty}(M)$. Applying Proposition \ref{prop_boundary_determination} to \eqref{eq_2_6}, we conclude that $A|_{\p M}=0$. Using this together with Stokes' formula,
\[
\int_M  \langle dw, \eta\rangle_g dV_g=\int_M w d^*\eta dV_g+\int_{\p M} \omega (\textbf{n}\eta),\ \omega\in C^\infty(M), \ \eta\in C^\infty(M,T^*M\otimes \C),
\]
where the $(2n-1)$-form $\textbf{n}\eta$ on the boundary  is the normal trace of $\eta$, see \cite[Proposition 2.1.2]{Schwarz_1995}, we obtain from  \eqref{eq_2_4} that 
\begin{equation}
\label{eq_2_7}
\int_M \big(3i d^*( A v^{(3)}) -(2i d^*(A) +V) v^{(3)}\big) v^{(1)}v^{(2)}dV_g=0,
\end{equation}
for all harmonic functions $v^{(l)}\in C^{\infty}(M)$, $l=1,2,3$.  Applying \cite[Theorem 1.1]{Guillarmou_Salo_Tzou_2019} 
together with the boundary determination result of \cite[Proposition 3.1]{Ma_Tzou_2021} to \eqref{eq_2_7}, we get 
\begin{equation}
\label{eq_2_8}
3i d^*( A v^{(3)}) -(2i d^*(A) +V) v^{(3)}=0,
\end{equation}
for every harmonic function $v^{(3)}\in C^{\infty}(M)$.  Using \eqref{eq_2_2_1}, we obtain from \eqref{eq_2_8} that 
\begin{equation}
\label{eq_2_9}
(i d^*(A) -V) v^{(3)}-3i \langle A, d v^{(3)}\rangle_g=0,
\end{equation}
for every harmonic function $v^{(3)}\in C^{\infty}(M)$. Letting $v^{(3)}=1$ in \eqref{eq_2_9}, we get 
\begin{equation}
\label{eq_2_10}
i d^*(A) -V=0,
\end{equation}
and therefore, 
\begin{equation}
\label{eq_2_11}
 \langle A, d v^{(3)}\rangle_g=0,
\end{equation}
for every harmonic function $v^{(3)}\in C^{\infty}(M)$.  Let $p\in M^{\text{int}}$ and by assumption (b), there exist $f_1,\dots, f_n\in \mathcal{O}(M)$ which form a complex coordinate system near $p$. Hence, $df_j(p)$, $d\overline{f_j}(p)$ is a basis for $T^*_pM\otimes \C$. Since on a K\"ahler manifold the Laplacian on functions satisfies 
\[
\Delta_g=d^*d=2\p^*\p=2\overline{\p }^*\overline{\p},
\]
see \cite[Lemma 2.1]{Guillarmou_Salo_Tzou_2019}, \cite[Theorem 8.6, p. 45]{Moroianu_2007}, we have that all functions $f_1,\dots, f_n$ as well as $\overline{f_1},\dots, \overline{f_n}$ are harmonic, and therefore, it follows from \eqref{eq_2_11} that 
\[
\langle A, d f_j\rangle_g(p)=0, \quad \langle A, d \overline{f_j}\rangle_g(p)=0.
\]
Hence, $A=0$,  and therefore, $A_1^{(1)}=A_1^{(2)}$ in $M$. It follows from \eqref{eq_2_10} that $V=0$, and therefore, $V^{(1)}_2=V^{(2)}_2$ in $M$. 

Let $m\ge 3$ and let us assume that 
\begin{equation}
\label{eq_2_11_1}
A_k^{(1)}=A_k^{(2)}, \quad k=1,\dots, m-2, \quad V^{(1)}_k=V^{(2)}_k,  \quad k=2,\dots, m-1. 
\end{equation}
To prove that $A_{m-1}^{(1)}=A_{m-1}^{(2)}$ and $V^{(1)}_m=V^{(2)}_m$, we shall use the $m$th order linearization of the Dirichlet--to--Neumann map. Such an $m$th order linearization with $m\ge 3$ is performed in \cite{KU20}, and combining  with \eqref{eq_2_11_1}, it  leads to the following integral identity,  
\begin{equation}
\label{eq_2_12}
\int_M \big((m+1)i \langle A, d(v^{(1)}\cdots v^{(m)})\rangle_g v^{(m+1)}-(mi d^*(A) +V) v^{(1)}\cdots v^{(m+1)} \big)dV_g=0,
\end{equation}
for all harmonic functions $v^{(l)}\in C^{\infty }(M)$, $l=1,\dots, m+1$, see \cite[Section 5]{KU20}. Here $A=A_{m-1}^{(1)}-A_{m-1}^{(2)}$ and $V=V_{m}^{(1)}-V_m^{(2)}$. 
Letting $v^{(1)}=\dots=v^{(m-2)}=1$ in \eqref{eq_2_12} and arguing as in the case $m=2$, we complete the proof of Theorem \ref{thm_main}.

\begin{rem}
Thanks to the density of products of two harmonic functions in the geometric setting of Theorem \ref{thm_main} established in \cite{Guillarmou_Salo_Tzou_2019}, we recover the nonlinear magnetic and electric potentials of the general form \eqref{eq_int_1_1} here. On the other hand,  in the case of conformally transversally anisotropic manifolds of real dimension $\ge 3$,  only the density of products of four harmonic functions is available, see \cite{Feizmohammadi_Oksanen}, \cite{LLLS}, \cite{KU20}, and therefore, the nonlinear magnetic and electric potentials of the form \eqref{eq_int_1_1} with $k\ge 2$ and $k\ge 3$, respectively,  were determined from the knowledge of the Dirichlet--to--Neumann map in \cite{KU20}.  
\end{rem}

\begin{appendix}
\section{Boundary determination of a $1$-form on a Riemannian manifold}
\label{app_boundary_determination}

When proving Theorem \ref{thm_main}, we need the following essentially known boundary determination result on a general compact Riemannian manifold with boundary, see \cite{Brown_Salo_2006}, \cite[Appendix A]{Krup_Uhlmann_magn_2018_advection}, \cite[Appendix C]{KU20}, \cite{Ma} for similar results. We present a proof for completeness and convenience of the reader. 
\begin{prop}
\label{prop_boundary_determination}
Let $(M,g)$ be a compact smooth Riemannian manifold of dimension $n\ge 2$ with smooth boundary. If $A\in C(M,T^*M\otimes \C)$ satisfies 
\begin{equation}
\label{eq_app_1}
\int_M\langle A, du\rangle_g \overline{u} dV_g=0,
\end{equation}
for every harmonic function $u\in C^\infty(M)$, then $A|_{\p M}=0$. 
\end{prop}

\begin{proof}
 In order to show that $A|_{\p M}=0$, we shall construct a suitable harmonic function $u\in C^\infty(M)$ to be used in the integral identity \eqref{eq_app_1}. When doing so, we shall use an explicit  family of functions $v_\lambda$, constructed  in  \cite{Brown_2001}, \cite{Brown_Salo_2006}, whose boundary values have a highly oscillatory behavior as $\lambda\to 0$, while becoming increasingly concentrated near a given point on the boundary of $M$.  We let $x_0\in \p M$ and we shall work in the boundary normal coordinates centered at $x_0$ so that in these coordinates, $x_0 =0$, the boundary $\p M$ is given by $\{x_n=0\}$, and $M^{\text{int}}$ is given by $\{x_n > 0\}$.   We have $T_{x_0}\p M=\R^{n-1}$, equipped with the Euclidean metric.  The unit tangent vector $\tau$ is then given by $\tau=(\tau',0)$ where $\tau'\in \R^{n-1}$, $|\tau'|=1$.   Associated to the tangent vector $\tau'$ is the covector $\sum_{\beta=1}^{n-1} g_{\alpha \beta}(0) \tau'_\beta=\tau'_\alpha\in T^*_{x_0}\p M$. 

Letting  $\frac{1}{3}\le \alpha\le \frac{1}{2}$ and following \cite{Brown_Salo_2006}, see also \cite[Appendix A]{Feizm_K_Oksan_Uhl},   we set 
\[
v_\lambda(x)= \lambda^{-\frac{\alpha(n-1)}{2}-\frac{1}{2}}\eta\bigg(\frac{x}{\lambda^{\alpha}}\bigg)e^{\frac{i}{\lambda}(\tau'\cdot x'+ ix_n)}, \quad 0<\lambda\ll 1,
\]
where $\eta\in C^\infty_0(\R^n;\R)$ is such that $\supp(\eta)$ is in a small neighborhood of $0$, and 
\[
\int_{\R^{n-1}}\eta(x',0)^2dx'=1.
\]
Here $\tau'$ is viewed as a covector. Thus, we have $v_\lambda\in C^\infty(M)$  with $\supp(v_\lambda)$ in  $\mathcal{O}(\lambda^{\alpha})$ neighborhood of $x_0=0$. A direct computation shows that 
\begin{equation}
\label{eq_app_4}
\|v_\lambda\|_{L^2(M)}=\mathcal{O}(1),  
\end{equation}
as $\lambda\to 0$, see also \cite[Appendix A, (A.8)]{Feizm_K_Oksan_Uhl}.  Furthermore, we have
\begin{equation}
\label{eq_app_11}
\|dv_\lambda\|_{L^2(M)}=\mathcal{O}(\lambda^{-1}),
\end{equation}
as $\lambda\to 0$,  see  \cite[Appendix C, bound (C.42)]{KU20}.

Following \cite{Brown_Salo_2006}, we set 
\begin{equation}
\label{eq_app_5}
u=v_\lambda+r,
\end{equation} 
where $r\in H^1_0(M^{\text{int}})$ is the unique solution to the Dirichlet problem, 
\begin{equation}
\label{eq_app_6}
\begin{cases}
-\Delta_g r=\Delta_g v_\lambda & \text{in}\quad M^{\text{int}},\\
r|_{\p M}=0.
\end{cases}
\end{equation} 
Boundary elliptic regularity implies $r\in C^\infty(M)$, and hence, $u\in C^\infty(M)$. Following \cite[Appendix A]{Feizm_K_Oksan_Uhl}, we fix $\alpha=1/3$. The following bound, proved in \cite[Appendix A, bound (A.15)]{Feizm_K_Oksan_Uhl}, will be needed here, 
\begin{equation}
\label{eq_app_7}
\|r\|_{L^2(M)}=\mathcal{O}(\lambda^{1/12}),
\end{equation}
as $\lambda\to 0$. The proof of \eqref{eq_app_7} relies on elliptic estimates for the Dirichlet problem for the Laplacian in Sobolev spaces of low regularity. We shall also need the following rough bound
\begin{equation}
\label{eq_app_10}
\|r\|_{H^1(M^{\text{int}})}=\mathcal{O}(\lambda^{-1/3}),
\end{equation}
as $\lambda\to 0$,
established in \cite[Appendix C, bound (C.41)]{KU20}.

Substituting $u$ into \eqref{eq_app_1}, and multiplying \eqref{eq_app_1} by $\lambda$, we get
\begin{equation}
\label{eq_app_12}
0= \lambda\int_{M} \langle A, dv_\lambda+dr\rangle_g(\overline{v_\lambda}+\overline{r})dV_g= \lambda( I_1+I_2+I_3),
\end{equation}
where 
\begin{align*}
I_1=\int_{M} \langle A, dv_\lambda\rangle_g \overline{v_\lambda}dV_g, \quad I_2=\int_{M} \langle A, dr \rangle_g(\overline{v_\lambda}+\overline{r})dV_g,\quad I_3= \int_{M} \langle A, dv_\lambda\rangle_g \overline{r}dV_g.
\end{align*}
It was computed in \cite[Appendix C]{KU20}, see bounds (C.44) and (C.45) there,  that 
\begin{equation}
\label{eq_app_13}
\lim_{\lambda\to 0}\lambda I_1=\frac{i}{2}\langle A(0),(\tau',i) \rangle.
\end{equation}
It follows from \eqref{eq_app_10}, \eqref{eq_app_4}, and \eqref{eq_app_7} that 
\begin{equation}
\label{eq_app_14}
\lambda |I_2|\le \mathcal{O}(\lambda)\|d r\|_{L^2(M)}\|v_\lambda+r\|_{L^2(M)}=\mathcal{O}(\lambda^{2/3}). 
\end{equation}
Using \eqref{eq_app_11} and \eqref{eq_app_7}, we get 
\begin{equation}
\label{eq_app_15}
\lambda |I_3|\le \mathcal{O}(\lambda)\|d v_\lambda\|_{L^2(M)}\|r\|_{L^2(M)}=\mathcal{O}(\lambda^{1/12}). 
\end{equation}
Passing to the limit $\lambda\to 0$ in \eqref{eq_app_12} and using \eqref{eq_app_13}, \eqref{eq_app_14}, \eqref{eq_app_15}, we obtain that $\langle A(0),(\tau',i) \rangle=0$, and arguing as in \cite[Appendix C]{KU20}, we get $A|_{\p M}=0$. This completes the proof of Proposition \ref{prop_boundary_determination}. 
 \end{proof}

\end{appendix}

\section*{Acknowledgements}

The research of K.K. is partially supported by the National Science Foundation (DMS 2109199). The research of G.U. is partially supported by NSF, a Walker Professorship at UW and a Si-Yuan Professorship at IAS, HKUST. The research of L.Y. is partially supported by the National Science Foundation (DMS 2109199).


\begin{thebibliography}{99}


\bibitem{Angulo_2017}
Angulo-Ardoy, P., \emph{On the set of metrics without local limiting Carleman weights}, 
Inverse Probl. Imaging \textbf{11} (2017), no. 1, 47--64. 

\bibitem{Angulo_Faraco_Guijarro_Ruiz_2016}
Angulo-Ardoy, P.,  Faraco, D.,  Guijarro, L.,  Ruiz, A., \emph{Obstructions to the existence of limiting Carleman weights},  Anal. PDE \textbf{9} (2016), no. 3, 575--595.

\bibitem{Assylbekov_Zhou}
Assylbekov, Y., Zhou, T., \emph{Direct and inverse problems for the nonlinear time-harmonic Maxwell equations in Kerr-type media}, J. Spectr. Theory \textbf{11} (2021), no. 1, 1--38.

\bibitem{Brown_2001}
Brown,  R., \emph{Recovering the conductivity at boundary from Dirichlet to Neumann map: a pointwise result}, J. Inverse Ill-Posed Probl.  \textbf{9} (2001), 567--574.

\bibitem{Brown_Salo_2006}
Brown, R.,  Salo, M., \emph{Identifiability at the boundary for first-order terms}, 
Appl. Anal. \textbf{85} (2006), no. 6-7, 735--749. 

\bibitem{Bukhgeim_2008}
Bukhgeim, A., \emph{Recovering a potential from Cauchy data in the two- dimensional case}, J. Inverse Ill-posed Probl. \textbf{16} (2008), 19--34.

\bibitem{CFa20}
C\^{a}rstea,  C., Feizmohammadi, A., \emph{An inverse boundary value problem for certain anisotropic quasilinear elliptic equations}, J. Differential Equations, \textbf{284} (2021), 318--349.


\bibitem{CFb20}
C\^{a}rstea, C.,   Feizmohammadi, A., \emph{A density property for tensor products of gradients of harmonic functions and applications}, preprint 2020, \textsf{https://arxiv.org/abs/2009.11217}. 

\bibitem{CFKKU_2021}
C\^{a}rstea, C.,  Feizmohammadi, A., Kian, Y.,  Krupchyk, K., Uhlmann, G., \emph{The Calder\'on inverse problem for isotropic quasilinear conductivities}, Adv. Math. \textbf{391} (2021), Paper No. 107956.




\bibitem{CNV_2019}
C\^{a}rstea, C., Nakamura, G., Vashisth, M., \emph{Reconstruction for the coefficients of a quasilinear elliptic partial differential equation}, Appl. Math. Lett. \textbf{98} (2019), 121--127.



\bibitem{Cekic}
Ceki\'c, M., \emph{The Calder\'on problem for connections},
Comm. Partial Differential Equations \textbf{42} (2017), no. 11, 1781--1836.



\bibitem{DKSaloU_2009} 
Dos Santos Ferreira, D.,  Kenig, C., Salo, M., Uhlmann, G.,  \emph{Limiting Carleman weights and anisotropic inverse problems},  Invent. Math. \textbf{178} (2009), no. 1, 119--171.


\bibitem{DKurylevLS_2016}
Dos Santos Ferreira, D., Kurylev, Y., Lassas, M., Salo, M., \emph{The Calder\'on problem in transversally anisotropic geometries},  J. Eur. Math. Soc. (JEMS) \textbf{18} (2016), no. 11, 2579--2626. 

\bibitem{Feizm_K_Oksan_Uhl}
Feizmohammadi, A., Krupchyk, K., Oksanen, L., Uhlmann, G., \emph{Reconstruction in the Calder\'on problem on conformally transversally anisotropic manifolds},  J. Funct. Anal. \textbf{281} (2021), no. 9, Paper No. 109191.

\bibitem{Feizmohammadi_Oksanen}
Feizmohammadi, A.,  Oksanen, L., \emph{An inverse problem for a semi-linear elliptic equation in Riemannian geometries}, J. Differential Equations \textbf{269} (2020), no. 6, 4683--4719.


\bibitem{Guillarmou_Tzou_2011_partial}
Guillarmou, C., Tzou, L., \emph{Calder\'on inverse problem with partial data on Riemann surfaces},  Duke Math. J. 158 (2011), no. 1, 83--120. 

\bibitem{Guillarmou_Tzou_2011}
Guillarmou, C., Tzou, L., \emph{Identification of a connection from Cauchy data on a Riemann surface with boundary}, Geom. Funct. Anal. \textbf{21} (2011), no. 2, 393--418.

\bibitem{Guillarmou_Tzou_2011_proc}
Guillarmou, C., Tzou, L., \emph{Calder\'on inverse problem for the Schr\"odinger operator on Riemann surfaces}, Proc. Centre Math. Appl. Austral. Nat. Univ., 44, Austral. Nat. Univ., Canberra, 2010, 129--141.



\bibitem{Guillarmou_Salo_Tzou_2019}
Guillarmou, C.,  Salo, M., Tzou, L., \emph{The linearized Calder\'on problem on complex manifolds}, 
Acta Math. Sin. (Engl. Ser.) \textbf{35} (2019), no. 6, 1043--1056. 


\bibitem{Isakov_93}
Isakov, V., \emph{On uniqueness in inverse problems for semilinear parabolic equations},  Arch. Rational Mech. Anal. \textbf{124} (1993), no. 1, 1--12.


\bibitem{KKU_conduc_partial}
Kian, Y., Krupchyk, K.,  Uhlmann, G.,  \emph{Partial data inverse problems for quasilinear conductivity equations}, preprint 2020,  \textsf{https://arxiv.org/abs/2010.11409}.

\bibitem{Krup_Uhlmann_2014}
Krupchyk, K., Uhlmann, G., \emph{Uniqueness in an inverse boundary problem for a magnetic Schr\"odinger operator with a bounded magnetic potential}, Comm. Math. Phys. \textbf{327} (2014), no. 3, 993--1009.

\bibitem{Krup_Uhlmann_magn_2018}
Krupchyk, K., Uhlmann, G., \emph{Inverse problems for magnetic Schr\"odinger operators in transversally anisotropic geometries}, Comm. Math. Phys. \textbf{361} (2018), no. 2, 525--582.

\bibitem{Krup_Uhlmann_magn_2018_advection}
Krupchyk, K., Uhlmann, G., \emph{Inverse problems for advection diffusion equations in admissible geometries}, Comm. Partial Differential Equations 43 (2018), no. 4, 585--615.

\bibitem{KU20_remark}
Krupchyk, K., Uhlmann, G., \emph{A remark on partial data inverse problems for semilinear elliptic equations}, Proc. Amer. Math. Soc. \textbf{148} (2020), no. 2, 681--685. 


\bibitem{KU20_MRL}
Krupchyk, K., Uhlmann, G., \emph{Partial data inverse problems for semilinear elliptic equations with gradient nonlinearities}, Math. Res. Lett. \textbf{27} (2020), no. 6, 1801--1824. 




\bibitem{KU20}
Krupchyk, K., Uhlmann, G., \emph{Inverse problems for nonlinear magnetic Schr\"odinger equations on conformally transversally anisotropic manifolds}, preprint 2020, \textsf{https://arxiv.org/abs/2009.05089}. 



\bibitem{Kurylev_Lassas_Uhlmann_2018}
Kurylev, Y., Lassas, M., Uhlmann, G., \emph{Inverse problems for Lorentzian manifolds and non-linear hyperbolic equations}, Invent. Math. \textbf{212} (2018), no. 3, 781--85.

\bibitem{Lai_Ting}
Lai, R.-Y.,  Zhou, T., \emph{Partial Data Inverse Problems for Nonlinear Magnetic Schr\"odinger Equations}, preprint 2020, \textsf{https://arxiv.org/abs/2007.02475}.


\bibitem{LLLS_partial}
Lassas, M., Liimatainen, T.,   Lin, Y-H.,  Salo, M., \emph{Partial data inverse problems and simultaneous recovery of boundary and coefficients for semilinear elliptic equations}, 
Rev. Mat. Iberoamericana \textbf{37} (2021), no. 4, 1553--1580.


\bibitem{LLLS}
Lassas, M., Liimatainen, T.,   Lin, Y-H.,  Salo, M., \emph{Inverse problems for elliptic equations with power type nonlinearities}, J. Math. Pures Appl. \textbf{145} (2021), 44--82.

\bibitem{Lassas_Uhlmann}
Lassas, M.,  Uhlmann, G., \emph{On determining a Riemannian manifold from the Dirichlet-to-Neumann map},  Ann. Sci. \'Ecole Norm. Sup. (4) \textbf{34} (2001), no. 5, 771--787.

\bibitem{Lassas_Taylor_Uhlmann}
Lassas, M., Taylor, M., Uhlmann, G., \emph{The Dirichlet--to--Neumann map for complete Riemannian manifolds with boundary},  Comm. Anal. Geom. 11 (2003), no. 2, 207--221. 


\bibitem{Lee_Uhlmann} 
 Lee, J.,  Uhlmann, G., \emph{Determining anisotropic real-analytic conductivities by boundary measurements}, 
Comm. Pure Appl. Math. \textbf{42} (1989), no. 8, 1097--1112. 

\bibitem{LLinSTyni}
Liimatainen, T.,   Lin, Y-H.,  Salo, M., Tyni, T., \emph{Inverse problems for elliptic equations with fractional power type nonlinearities}, J. Differential Equations, to appear. 

\bibitem{LS_2012}  
Liimatainen, T.,  Salo, M., \emph{Nowhere conformally homogeneous manifolds and limiting Carleman weights}, Inverse Probl. Imaging \textbf{6} (2012), no. 3, 523--530. 


\bibitem{Ma}
Ma, Y., \emph{A note on the partial data inverse problems for a nonlinear magnetic Schr\"odinger operator on Riemann surface}, preprint 2020, \textsf{https://arxiv.org/abs/2010.14180}. 

\bibitem{Ma_Tzou_2021}
Ma, Y., Tzou, L., \emph{Semilinear Calder\'on problem on Stein manifolds with K\"ahler metric},  Bull. Aust. Math. Soc. \textbf{103} (2021), no. 1, 132--144.

\bibitem{Moroianu_2007}
Moroianu,  A., \emph{Lectures on K\"ahler geometry},  London Mathematical Society Student Texts \textbf{69}, Cambridge University Press, 2007.

\bibitem{MU20}
Mu\~{n}oz, C.,  Uhlmann, G., \emph{The Calder\'{o}n problem for quasilinear elliptic equations},
 Ann. Inst. H. Poincar\'{e} Anal. Non Lin\'{e}aire, \textbf{37} (2020), no. 5, 1143--1166.



\bibitem{Nakamura_Sun_Uhlmann}
Nakamura, G.,  Sun, Z., Uhlmann, G., \emph{Global identifiability for an inverse problem for the Schr\"odinger equation in a magnetic field},  Math. Ann. \textbf{303} (1995), no. 3, 377--388.




\bibitem{Schwarz_1995}
Schwarz, G., \emph{Hodge decomposition--a method for solving boundary value problems}, Lecture Notes in Mathematics, 1607. Springer-Verlag, Berlin, 1995. 

\bibitem{Sh19}
Shankar, R., \emph{Recovering a quasilinear conductivity from boundary measurements},
 Inverse problems, \textbf{37} (2019), 015014.


\bibitem{Sun_96}
Sun, Z., \emph{On a quasilinear inverse boundary value problem},  Math. Z. \textbf{221} (1996), no. 2, 293--305.


\bibitem{Sun_Uhlm_97}
Sun, Z.,  Uhlmann, G., \emph{Inverse problems in quasilinear anisotropic media},
Amer. J. Math. \textbf{119} (1997), no. 4, 771--797.

\bibitem{Sylvester_Uhlmann_1987}
Sylvester, J., Uhlmann, G., \emph{A global uniqueness theorem for an inverse boundary value problem}, Ann. of Math. (2)
\textbf{125} (1987), no. 1, 153--169.


\end{thebibliography}
\end{document}